\documentclass[12pt]{amsart}
\usepackage[pdftex]{hyperref}
\usepackage{tikz}
\usepackage{amssymb}
\usepackage{multirow}
\allowdisplaybreaks

\let\checkedref\ref
\let\checkedeqref\eqref

\def\even{{\hbox{\scriptsize\rm even}}}
 
\def\tb{\mathop{\rm tb}\nolimits}

\def\size{\mathop{\rm size}\nolimits}

\def\sset{\subseteq}

\let\cong\equiv

\def\Z{\mathbb{Z}}

\def\pspace{p\kern.1em}

\newtheorem{theorem}{Theorem}
\newtheorem{lemma}[theorem]{Lemma}

\theoremstyle{remark}
\newtheorem{remark}[theorem]{Remark}

\newtheorem{example}[theorem]{Example}
\numberwithin{theorem}{section}
\numberwithin{equation}{section}

\begin{document}

\title{Best play in Dots~\& Boxes endgames}
\author{Daniel Allcock}
\thanks{Supported by Simons Foundation Collaboration Grant 429818}
\address{Department of Mathematics\\University of Texas, Austin}
\email{allcock@math.utexas.edu}
\urladdr{http://www.math.utexas.edu/\textasciitilde allcock}
\date{November 26, 2018}

\begin{abstract}
    We give very simple algorithms for best play in 
    the simplest kind of
    Dots~\& Boxes endgames: those that consist entirely of loops
    and long chains.  In every such endgame we compute
    the margin of victory, 
    assuming both players
    maximize the numbers of boxes they capture, and specify a move
    that leads to that result.
    We  improve on results by Buzzard and Ciere
    \cite{BC} on the same problem: our algorithms
    examine only the current position and do not need to consider
    the game tree at all.
\end{abstract}

\maketitle

\section{Introduction}
\label{SecIntro}

Dots~\& Boxes is one of the few pen-and-paper games that
has a rich mathematical theory
and is also played at a high level
by many people without any special interest in mathematics.  
There are several websites where one can play online against 
other people,
such as \texttt{littlegolem.net}
and \texttt{yourturnmyturn.com}.
Go shares these properties, but mathematics  
has more limited application in actual play~\cite{Berlekamp-Wolfe}.

We give the 
rules and other background in section~\checkedref{SecDnB}.
In brief, one begins with a grid of dots.  When it is your turn, you
join two  adjacent dots.  If you complete a box this way
then you move again immediately.  You continue in this way until you
move without completing
a box (when play passes to your opponent) or you  
complete the last box (when the game ends).
Whoever completes more boxes wins.
See section~\checkedref{SecDnB} for details and for some
standard elements of play like  \emph{loops} and \emph{long chains}, 
the \emph{parity rule} and the \emph{hard-hearted handout}.
For further background we recommend 
\cite{Berlekamp} or \cite{WinningWays}.

The first 
contribution of mathematics to Dots~\& Boxes is the idea of 
moving as though playing a simpler game called Nimstring. 
The rules are the same except for who wins: in Nimstring,
the winner is whoever completes the
last box.  This sounds like a radical change, but 
good Nimstring moves are often good Dots~\& Boxes moves.
As the name suggests, Nimstring is amenable to the Sprague-Grundy
theory for Nim-like games \cite[Ch.~16]{WinningWays}.
This leads to a trick 
called the parity rule.  If neither player knows
it then the game proceeds mostly randomly; if only 
one knows it then she will win;
if both know it then the game becomes challenging and interesting.  
The interest
and challenge come from understanding the positions where there 
is a winning move,
even when the Nimstring perspective suggests the game is lost.   

An expert player who expects to lose the Nimstring game will try 
to steer the game toward
such a position.
Typically he does this by aiming for an endgame with
many $3$-chain and $4$- and $6$-loops.  
Therefore high-level play requires an understanding of such positions.
This paper gives a complete analysis of  endgames 
that consist entirely of 
loops and long chains.  We restrict attention 
to such positions unless otherwise indicated.
We were inspired by the study of these endgames by
Buzzard and Ciere \cite{BC}.  While their algorithms 
for determining the value of a game
and an optimal move
require only linear time, they 
involve trees of cases and subcases, and also a limited form of recursion.
Our algorithms are simpler.  

\medskip
If it is your turn and you must choose a loop or 
long chain to open (play in) then we call you the
\emph{opener}.
Once you choose this component~$C$, it does not
matter where in~$C$ you play. Your opponent will
either take all those boxes, or reply with the hard-hearted
handout (see below), and where you play in~$C$ has no
effect on these options. 

Nimstring suggests
that you have lost.  We will see
that you cannot win the endgame, in the sense
that you cannot capture more of the remaining boxes than a
skilled
opponent.
But if you 
captured more boxes than she did,
before reaching the endgame, 
and you lose the endgame by only a little, then
 you may
still be able to win.
The strategy in theorem~\checkedref{ThmOpener} is guaranteed to lose
the endgame by as little as possible, and therefore
gives your best chance of winning.  
Your opponent is called the \emph{controller},
for reasons explained below.

A chain or loop of length $N$ is called an $N$-chain or
$N$-loop, and often indicated by simply writing $N$, attaching
a subscript $\ell$ to indicate a loop.  
We use additive notation in the obvious way, for example
$G=3+4_\ell+8_\ell$
means that $G$ consists
of a $3$-chain, a $4$-loop, an $8$-loop and possibly some
already-claimed boxes:
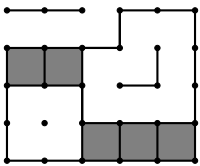
\def\SCALE{.5}
\def\dotsize{.07}
$$
\begin{tikzpicture}[xscale=\SCALE,yscale=\SCALE]
    \fill[gray](0,2)rectangle(2,3);
    \fill[gray](2,0)rectangle(5,1);
    \draw[thick](0,4)--(2,4);
    \draw[thick](0,3)--(3,3)--(3,4);
    \draw[thick](0,0)--(0,2)--(2,2)--(2,0)--cycle;
    \draw[thick](5,1)--(5,4)--(3,4)--(3,3)--(2,3)--(2,1)--cycle;
    \draw[thick](3,2)--(4,2)--(4,3);
    \draw[thick](0,2)--(0,3);
    \draw[thick](1,2)--(1,3);
    \draw[thick](3,0)--(3,1);
    \draw[thick](4,0)--(4,1);
    \draw[thick](5,1)--(5,0)--(2,0);
    \draw[fill=black](0,0)circle(\dotsize);
    \draw[fill=black](0,1)circle(\dotsize);
    \draw[fill=black](0,2)circle(\dotsize);
    \draw[fill=black](0,3)circle(\dotsize);
    \draw[fill=black](0,4)circle(\dotsize);
    \draw[fill=black](1,0)circle(\dotsize);
    \draw[fill=black](1,1)circle(\dotsize);
    \draw[fill=black](1,2)circle(\dotsize);
    \draw[fill=black](1,3)circle(\dotsize);
    \draw[fill=black](1,4)circle(\dotsize);
    \draw[fill=black](2,0)circle(\dotsize);
    \draw[fill=black](2,1)circle(\dotsize);
    \draw[fill=black](2,2)circle(\dotsize);
    \draw[fill=black](2,3)circle(\dotsize);
    \draw[fill=black](2,4)circle(\dotsize);
    \draw[fill=black](3,0)circle(\dotsize);
    \draw[fill=black](3,1)circle(\dotsize);
    \draw[fill=black](3,2)circle(\dotsize);
    \draw[fill=black](3,3)circle(\dotsize);
    \draw[fill=black](3,4)circle(\dotsize);
    \draw[fill=black](4,0)circle(\dotsize);
    \draw[fill=black](4,1)circle(\dotsize);
    \draw[fill=black](4,2)circle(\dotsize);
    \draw[fill=black](4,3)circle(\dotsize);
    \draw[fill=black](4,4)circle(\dotsize);
    \draw[fill=black](5,0)circle(\dotsize);
    \draw[fill=black](5,1)circle(\dotsize);
    \draw[fill=black](5,2)circle(\dotsize);
    \draw[fill=black](5,3)circle(\dotsize);
    \draw[fill=black](5,4)circle(\dotsize);
\end{tikzpicture}
$$
In actual play the claimed boxes (shaded here) would be marked
with the players' initials.  
We ignore
such boxes
when discussing $G$ because they do not affect play.
The \emph{value}
$v(G)$ of any endgame~$G$ means the margin
by which the controller will beat the opener, assuming 
that they enter the endgame
with a tie score and then play optimally.  
So $v(3+4_\ell+8_\ell)=1$ means that the opener can win if
and only if he earned 
at least
a $2$-box advantage during whatever play led to this endgame.
To lose this endgame by only one box, open the
$4$-loop first and  the $8$-loop second.  

The  size of $G$  means the number of (unclaimed) boxes in~$G$, and
the \emph{controlled value} $c(G)$
is defined below
and explained in section~\checkedref{SecDnB}.
For now it is enough to know that you can compute $c(G)$ 
quickly in your head.
The 
\emph{standard move} 
means to open a $3$-chain if one is present, 
otherwise a shortest loop if a loop is present, and
otherwise a shortest chain.
See section~\checkedref{SecOpener} for the proof of the following result.

\begin{theorem}[Opener strategy]
    \label{ThmOpener}
    Suppose $G$ is a nonempty Dots~\& Boxes position that consists 
    of loops and long chains.  
    In each of the following cases,
    opening the shortest loop is optimal:
    \begin{enumerate}
        \item
            \label{ItemMaster3PlusLoops}
            \leavevmode
            \rlap{$c(G)\geq2$}%
            \phantom{$c(G)\in\{0,\pm1\}$}
            and $G=3_{\phantom{\ell}}+\hbox{\rm(one or more loops)};$
        \item
            \label{ItemMasterCAtLeastMinus1}
            \leavevmode
            \hbox{$c(G)\in\{0,\pm1\}$} 
            and  
            $G=4_\ell+\hbox{\rm(anything except $3+3+3$)};$
        \item
            \label{ItemMasterSize3Mod4}
            \leavevmode
            \rlap{$c(G)\leq-2$}%
            \phantom{$c(G)\in\{0,\pm1\}$}
            and 
            $G=4_\ell+3+H$, where $4|\size(H)$ and $H$ has no $3$-chains.
    \end{enumerate}
    In all other cases the standard move is optimal.
\end{theorem}

The ``phase transition'' when $c(G)$ changes from $1$ to~$2$ 
was discovered by Berlekamp.  
The transition from $-2$ to $-1$ appears to be new.  The proofs
of this theorem and those below involve much case analysis,
but the results  can be assembled afterwards into fairly simple 
conclusions.

Berlekamp introduced the controlled value $c(G)$
because it is easy to compute and 
carries a lot of information about $v(G)$.  
For example, they are equal if $c(G)\geq2$ 
(see theorem~\checkedref{ThmCAtLeast2}).
The definition is
\begin{equation}
    \label{EqDefControlledValue}
    c(G) = \size(G) - 4(\hbox{\# long chains}) -8(\hbox{\# loops})
    +\tb(G)
\end{equation}
where $\tb(G)$  
is called the \emph{terminal bonus} of $G$ and defined
by
\begin{equation}
  \label{EqTB}
  \tb(G)=
  \begin{cases}
    $0$&\hbox{if $G$ is empty}
    \\
    $8$&\hbox{if $G=(\hbox{one or more loops})$}
    \\
      $6$&\hbox{if $G=(\hbox{one or more loops})
      +(\hbox{one or more $3$-chains})$}
    \\
    $4$&\hbox{otherwise}
  \end{cases}
\end{equation}
See section~\checkedref{SecDnB} for why $c(G)$ is called
the controlled value.

\begin{example}
    Suppose $G$ consists of five $3$-chains, a $4$-loop and an $8$-loop.
    We have
    $c(G)=27-4\cdot5-8\cdot2+6=-3$.  
    The theorem says that opening the components in the
    order $3$, $3$, $4_\ell$, $3$, $3$, $8_\ell$,~$3$ is optimal.
    We used a computer to check this, and also that the only other
    optimal line of play is
    $3$, $4_\ell$, $3$, $3$, $3$, $8_\ell$,~$3$.
    This shows that there is a certain subtlety to the order in
    which one must open the components, and suggests that theorem~\ref{ThmOpener}
    may be the simplest strategy possible.
\end{example}

If you are the controller, and the  opener has just
opened a loop or long chain, then the decision you face is whether
to keep control or give it up.  \emph{Giving up control} means that you 
take all of the boxes in the just-opened component.
Unless that ends the game
you must move again, so you become the opener and your opponent
the controller.  As the new opener, you can use 
theorem~\checkedref{ThmOpener} to choose which component to open.
\emph{Keeping control} means that 
you take all but a few 
of the boxes in the opened component
($4$ for a loop or $2$ for a long chain).  This is called 
the hard-hearted handout: a handout because you
are giving 
your opponent some boxes, and hard-hearted because after
he takes them he must open the next component.  
In this case he remains the opener and you remain
the controller.

Now we consider the situation where the opener has
just opened component $C$  of a position $G$.
It is easy to see that the controller should
keep control if $v(G-C)$ is larger than the number of boxes ($2$ or~$4$)
given away in the hard-hearted handout,  and give up control if $v(G-C)$
is smaller.  She may choose either option in case of equality.
We  use subtractive  notation in the obvious way: $G-C$ means
the position got from $G$ by removing~$C$.
So the following
well-known result gives an optimal strategy for the controller.

\begin{theorem}[Controller strategy]
    \label{ThmController}
If the opener has just opened a component $C$ of $G$,
then the following gives an optimal move for the controller.
Keep control
if $C$ is a loop and $c(G-C)>4$, or if $C$ is chain and $v(G-C)>2$;
otherwise give up control.
\qed
\end{theorem}

 To use this strategy the controller must be
able to recognize when $v(G-C)>2$ or~$4$.
It is very easy to recognize when
$v(G-C)>4$ because this is equivalent to $c(G-C)>4$.
This and 
the following result are corollaries of 
theorem~\ref{ThmValuesExplicit}, which gives all all
values explicitly.

\begin{theorem}[Values${}>2$]
  \label{ThmVbiggerThan2}
  We have $v(G)>2$ if and only if: 
  either $c(G)>2$, or else $G$ satisfies
  one of the two alternatives

  $G$ has exactly one $3$-chain and $\size(G)\cong3$ {\rm mod}~$4$

  \nopagebreak
  $G$ has no $3$-chains and $\size(G)\not\cong2$ {\rm mod}~$4$

  \noindent and one of the two alternatives

  $c(G)+4f(G)>2$ and $c(G)\cong\pm3$ or $4$ {\rm mod}~$8$

  \nopagebreak
  $c(G)+4f(G)<2$ and $f(G)$ is even

  \noindent
  where $f(G)$ is the number of $4$-loops in~$G$.
  \qed
\end{theorem}

In particular,
if two or more $3$-chains are present then $v(G)>2$ if and only
if $c(G)>2$.  
The same holds if just one 
is present and $\size(G)\not\cong3$ mod~$4$.
In the remaining cases we prefer  
theorem~\ref{ThmValuesProcedural}\eqref{ItemValProcOperators}
to the complicated
second pair of alternatives.  
It gives $v(G)$ as the result of a simple process applied
to a simple starting number, and is much easier to remember.

Section~\checkedref{SecMidgame}  gives some consequences for mid-game
strategy.  For example, 
the player who is opener when the endgame begins will
probably lose if he has only a 
$1$-box advantage, but very likely win with
a $2$-box advantage.
This assumes an 
odd${}\times{}$odd board and that enough $3$-chains and
$4$- and $6$-loops are
present to make $c(G)<2$.  
Players often sacrifice a box or two during the midgame to
create the ``right'' number  of long chains.  (The goal
is to be controller in the endgame, because in Nimstring
the controller always wins. See
the parity rule in section~\ref{SecDnB}.)  To first approximation
our results show that
sacrificing one box is safe but sacrificing two is suicide.

The natural next endgames to consider
have components more complicated than loops and chains, such as
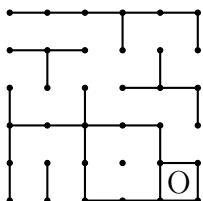
\def\dotsize{.07}
\def\GRID{%
    \draw[fill=black](0,0)circle(\dotsize);
    \draw[fill=black](0,1)circle(\dotsize);
    \draw[fill=black](0,2)circle(\dotsize);
    \draw[fill=black](0,3)circle(\dotsize);
    \draw[fill=black](0,4)circle(\dotsize);
    \draw[fill=black](0,5)circle(\dotsize);
    \draw[fill=black](1,0)circle(\dotsize);
    \draw[fill=black](1,1)circle(\dotsize);
    \draw[fill=black](1,2)circle(\dotsize);
    \draw[fill=black](1,3)circle(\dotsize);
    \draw[fill=black](1,4)circle(\dotsize);
    \draw[fill=black](1,5)circle(\dotsize);
    \draw[fill=black](2,0)circle(\dotsize);
    \draw[fill=black](2,1)circle(\dotsize);
    \draw[fill=black](2,2)circle(\dotsize);
    \draw[fill=black](2,3)circle(\dotsize);
    \draw[fill=black](2,4)circle(\dotsize);
    \draw[fill=black](2,5)circle(\dotsize);
    \draw[fill=black](3,0)circle(\dotsize);
    \draw[fill=black](3,1)circle(\dotsize);
    \draw[fill=black](3,2)circle(\dotsize);
    \draw[fill=black](3,3)circle(\dotsize);
    \draw[fill=black](3,4)circle(\dotsize);
    \draw[fill=black](3,5)circle(\dotsize);
    \draw[fill=black](4,0)circle(\dotsize);
    \draw[fill=black](4,1)circle(\dotsize);
    \draw[fill=black](4,2)circle(\dotsize);
    \draw[fill=black](4,3)circle(\dotsize);
    \draw[fill=black](4,4)circle(\dotsize);
    \draw[fill=black](4,5)circle(\dotsize);
    \draw[fill=black](5,0)circle(\dotsize);
    \draw[fill=black](5,1)circle(\dotsize);
    \draw[fill=black](5,2)circle(\dotsize);
    \draw[fill=black](5,3)circle(\dotsize);
    \draw[fill=black](5,4)circle(\dotsize);
    \draw[fill=black](5,5)circle(\dotsize);
}
\def\SCALE{.5}
\def\COLA{7.5}
\def\COLB{15}
\def\HT{3.25}
\begin{center}
    \begin{tikzpicture}[baseline=2.5,xscale=\SCALE,yscale=\SCALE]
        \draw[thick](0,0)--(0,3);
        \draw[thick](1,0)--(1,1);
        \draw[thick](0,2)--(2,2)--(2,0)--(5,0)--(5,1)--(4,1);
        \draw[thick](2,3)--(2,2)--(4,2)--(4,0);
        \draw[thick](0,4)--(2,4);
        \draw[thick](1,3)--(1,4);
        \draw[thick](0,5)--(5,5)--(5,4);
        \draw[thick](3,5)--(3,4);
        \draw[thick](3,3)--(5,3)--(5,2);
        \draw[thick](4,4)--(4,3);
        \draw(4.5,0.5)node{O};
        \GRID
    \end{tikzpicture}%
\end{center}
In this example every move is loony (ie, loses the Nimstring
game), but if the controller always keeps control then
the opener O can win.  She 
starts this endgame with a $1$-box advantage and can choose to play as
though in $4_\ell+4+4+4+8$, which has value~$0$.
One could also allow loops to have odd length.  It is
conceivable that our strategy for the opener remains valid.
But  length~$7$ loops definitely do complicate some
intermediate results, as in \cite[remark~16]{BC}, and
length~$5$ loops are probably worse.

\medskip
We are grateful to Kevin 
Buzzard and Michael Burton for interesting and helpful conversations.
In particular, this 
paper would not exist without the prior work
of Buzzard and Ciere \cite{BC}. 
We also remark that William Fraser  announced on the games
website {\tt littlegolem.net} that on April 28, 2017
he finished the calculations needed
for his program {\it The Shark}
to play the
$5\times5$ game perfectly.  
One can play against it on this
website, but  no further details seem to be public.

\section{Dots~\& Boxes}
\label{SecDnB}

\noindent
In the first part of
this section we review the rules of Dots~\& Boxes and 
some standard elements of good play.
The main references are \cite[Ch.~16]{WinningWays}
and  \cite{Berlekamp}.  
In the second part we discuss some more-technical
material that we will need.

The game begins with a grid of dots.  
At the end of the game, each pair of (horizontally or vertically) adjacent
dots will be joined by an edge, making a grid of boxes.  The board size is
typically $5\times5$, meaning $5$ boxes
by $5$ boxes.
Odd${}\times{}$odd boards are best because ties
are impossible.
After the game, each box will belong to one player
or the other, and whoever has the most
boxes will win.
On her turn,
a player moves by drawing a line connecting two adjacent dots that are not yet joined.  
Any boxes completed by this line then belong to her, and if at least one box was
completed then she moves again immediately.  She continues in this way
until she moves without completing a box
(when it becomes the other player's turn) or
she completes the grid (ending the game).  
In particular, a player's turn may consist of  more than one move.  
It often happens that 
the segment completing a box
also completes the third
edge of another box.  In this case the extra move
earned by completing the first 
box enables her to complete the second box,
which in turn might allow her to complete a third
box, and so on.
In this way a player may capture an entire chain of boxes in
a single turn.

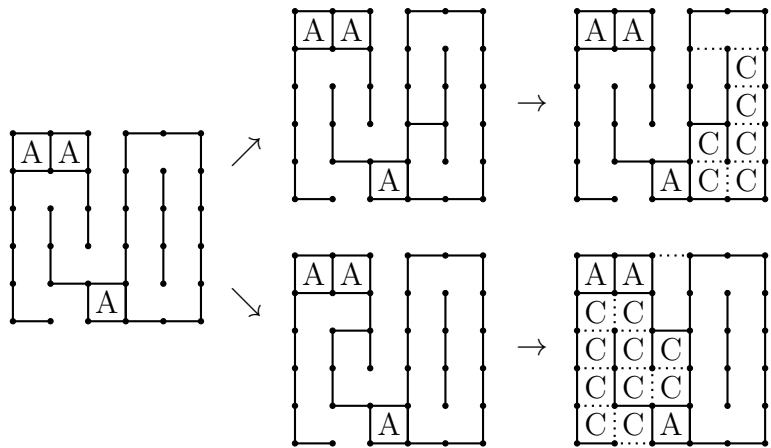
\begin{figure}
    \begin{tikzpicture}[xscale=\SCALE,yscale=\SCALE]
        \draw(0,0)node{%
        \begin{tikzpicture}[baseline=2.5,xscale=\SCALE,yscale=\SCALE]
            \draw[thick](3,0)--(3,5)--(5,5)--(5,0)--cycle;
            \draw[thick](4,1)--(4,4);
            \draw[thick](1,4)--(1,5)--(2,5)--(2,4)--(0,4)--(0,5)--(1,5);
            \draw(0.5,4.5)node{A};
            \draw(1.5,4.5)node{A};
            \draw[thick](2,0)--(2,1)--(3,1)--(3,0)--cycle;
            \draw(2.5,0.5)node{A};
            \draw[thick](2,2)--(2,4);
            \draw[thick](1,3)--(1,1)--(2,1);
            \draw[thick](0,4)--(0,0)--(1,0);
            \GRID
        \end{tikzpicture}%
            };%
        \draw(\COLA,\HT)node{%
            \begin{tikzpicture}[xscale=\SCALE,yscale=\SCALE]
                \draw[thick](3,0)--(3,5)--(5,5)--(5,0)--cycle;
                \draw[thick](4,1)--(4,4);
                \draw[thick](1,4)--(1,5)--(2,5)--(2,4)--(0,4)--(0,5)--(1,5);
                \draw(0.5,4.5)node{A};
                \draw(1.5,4.5)node{A};
                \draw[thick](2,0)--(2,1)--(3,1)--(3,0)--cycle;
                \draw(2.5,0.5)node{A};
                \draw[thick](2,2)--(2,4);
                \draw[thick](1,3)--(1,1)--(2,1);
                \draw[thick](0,4)--(0,0)--(1,0);
                \draw[thick](3,2)--(4,2);
            \GRID
            \end{tikzpicture}
            };%
        \draw(\COLA,-\HT)node{%
        \begin{tikzpicture}[xscale=\SCALE,yscale=\SCALE]
            \draw[thick](3,0)--(3,5)--(5,5)--(5,0)--cycle;
            \draw[thick](4,1)--(4,4);
            \draw[thick](1,4)--(1,5)--(2,5)--(2,4)--(0,4)--(0,5)--(1,5);
            \draw(0.5,4.5)node{A};
            \draw(1.5,4.5)node{A};
            \draw[thick](2,0)--(2,1)--(3,1)--(3,0)--cycle;
            \draw(2.5,0.5)node{A};
            \draw[thick](2,2)--(2,4);
            \draw[thick](1,3)--(1,1)--(2,1);
            \draw[thick](0,4)--(0,0)--(1,0);
            \draw[thick](1,3)--(2,3);
            \GRID
        \end{tikzpicture}
            };%
        \draw(\COLB,\HT)node{%
        \begin{tikzpicture}[xscale=\SCALE,yscale=\SCALE]
            \draw[thick](3,0)--(3,5)--(5,5)--(5,0)--cycle;
            \draw[thick](4,1)--(4,4);
            \draw[thick](1,4)--(1,5)--(2,5)--(2,4)--(0,4)--(0,5)--(1,5);
            \draw(0.5,4.5)node{A};
            \draw(1.5,4.5)node{A};
            \draw[thick](2,0)--(2,1)--(3,1)--(3,0)--cycle;
            \draw(2.5,0.5)node{A};
            \draw[thick](2,2)--(2,4);
            \draw[thick](1,3)--(1,1)--(2,1);
            \draw[thick](0,4)--(0,0)--(1,0);
            \draw[thick](3,2)--(4,2);
            \draw[thick,dotted](3,4)--(5,4);
            \draw[thick,dotted](4,3)--(5,3);
            \draw[thick,dotted](4,2)--(5,2);
            \draw[thick,dotted](3,1)--(5,1);
            \draw[thick,dotted](4,0)--(4,1);
            \draw(3.5,1.5)node{C};
            \draw(3.5,0.5)node{C};
            \draw(4.5,0.5)node{C};
            \draw(4.5,1.5)node{C};
            \draw(4.5,2.5)node{C};
            \draw(4.5,3.5)node{C};
            \GRID
        \end{tikzpicture}
            };%
        \draw(\COLB,-\HT)node{%
        \begin{tikzpicture}[xscale=\SCALE,yscale=\SCALE]
            \draw[thick](3,0)--(3,5)--(5,5)--(5,0)--cycle;
            \draw[thick](4,1)--(4,4);
            \draw[thick](1,4)--(1,5)--(2,5)--(2,4)--(0,4)--(0,5)--(1,5);
            \draw(0.5,4.5)node{A};
            \draw(1.5,4.5)node{A};
            \draw[thick](2,0)--(2,1)--(3,1)--(3,0)--cycle;
            \draw(2.5,0.5)node{A};
            \draw[thick](2,2)--(2,4);
            \draw[thick](1,3)--(1,1)--(2,1);
            \draw[thick](0,4)--(0,0)--(1,0);
            \draw[thick](2,3)--(3,3);
            \draw[thick,dotted](2,5)--(3,5);
            \draw[thick,dotted](0,3)--(2,3);
            \draw[thick,dotted](0,2)--(3,2);
            \draw[thick,dotted](0,1)--(1,1)--(1,0)--(2,0);
            \draw[thick,dotted](2,1)--(2,2);
            \draw[thick,dotted](1,3)--(1,4);
            \draw(0.5,3.5)node{C};
            \draw(1.5,3.5)node{C};
            \draw(0.5,2.5)node{C};
            \draw(1.5,2.5)node{C};
            \draw(2.5,2.5)node{C};
            \draw(0.5,1.5)node{C};
            \draw(1.5,1.5)node{C};
            \draw(2.5,1.5)node{C};
            \draw(0.5,0.5)node{C};
            \draw(1.5,0.5)node{C};
            \GRID
        \end{tikzpicture}
            };%
            \draw(3.7,2)node{$\nearrow$};
            \draw(3.7,-2)node{$\searrow$};
            \draw(11.3,\HT)node{$\to$};
            \draw(11.3,-\HT)node{$\to$};
    \end{tikzpicture}
    \caption{A Dots~\& Boxes position consisting of a $12$-chain
        and a $10$-loop, and two lines of play demonstrating the hard-hearted
    handout}
    \label{FigExamples}
\end{figure}

A sample position appears in figure~\checkedref{FigExamples}.  
During earlier play the author completed
three boxes and placed his initial inside them to mark them as his.  
This makes it looks like he is winning, but 
it is his turn and no good move is available.  
Any move in the loop on the right lets his opponent C
(for ``controller'') 
capture all $10$ boxes there.
Similarly, any move in the 
chain on the left lets her
capture all $12$ of its boxes.  
When a child, the author would have moved in the loop,
expecting her  to capture those $10$ boxes and then open
the $12$-chain for him to capture.
He would
win, $15$-$10$.  

Unfortunately for the author, C knows the 
\emph{hard-hearted handout:} when he
offers her the $10$ boxes, she takes all but~$4$, for example as shown
in the top line of play.
Dotted lines indicate her moves.  Note the $2\times1$
and $1\times2$
rectangles that she could have claimed but chose not to.
Your poor author wins those $4$ boxes but then must move again, opening
the $12$-chain for C to capture entire.
Therefore C wins, $18$-$7$.  
If the author opens the chain instead of the loop
then C takes
all but two of its boxes, as shown in the bottom line of play.  
This is the chain version
of the hard-hearted handout.  In this case the author would lose
$20$-$5$.
(A few people play with a rule that forbids the hard-hearted handout:
  if one can complete a box then one must.  
  But forced greedy strategies make games dull.)

\def\GRID{%
    \draw[fill=black](0,0)circle(\dotsize);
    \draw[fill=black](0,1)circle(\dotsize);
    \draw[fill=black](1,0)circle(\dotsize);
    \draw[fill=black](1,1)circle(\dotsize);
    \draw[fill=black](2,0)circle(\dotsize);
    \draw[fill=black](2,1)circle(\dotsize);
}
\def\COLA{0}
\def\COLB{5}
\def\COLC{8.5}
\def\COLD{13.5}
\def\ROW{2.5}
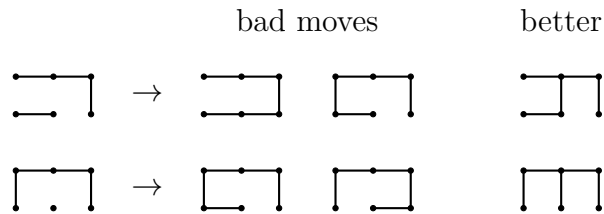
\begin{figure}
    \begin{tikzpicture}[xscale=\SCALE,yscale=\SCALE]
        \draw(\COLA,\ROW)node{%
                \begin{tikzpicture}[baseline=2.5,xscale=\SCALE,yscale=\SCALE]
                    \useasboundingbox(0,0)rectangle(2,1);
                    \draw[thick](0,1)--(2,1)--(2,0);
                    \draw[thick](0,0)--(1,0);
                    \GRID
                \end{tikzpicture}%
            };%
        \draw(\COLB,\ROW)node{%
                \begin{tikzpicture}[baseline=2.5,xscale=\SCALE,yscale=\SCALE]
                    \useasboundingbox(0,0)rectangle(2,1);
                    \draw[thick](0,1)--(2,1)--(2,0);
                    \draw[thick](0,0)--(2,0);
                    \GRID
                \end{tikzpicture}%
            };%
        \draw(\COLC,\ROW)node{%
                \begin{tikzpicture}[baseline=2.5,xscale=\SCALE,yscale=\SCALE]
                    \useasboundingbox(0,0)rectangle(2,1);
                    \draw[thick](0,1)--(2,1)--(2,0);
                    \draw[thick](1,0)--(0,0)--(0,1);
                    \GRID
                \end{tikzpicture}%
            };%
        \draw(\COLD,\ROW)node{%
                \begin{tikzpicture}[baseline=2.5,xscale=\SCALE,yscale=\SCALE]
                    \useasboundingbox(0,0)rectangle(2,1);
                    \draw[thick](0,1)--(2,1)--(2,0);
                    \draw[thick](0,0)--(1,0)--(1,1);
                    \GRID
                \end{tikzpicture}%
            };%
        \draw(\COLA,0)node{%
                \begin{tikzpicture}[baseline=2.5,xscale=\SCALE,yscale=\SCALE]
                    \useasboundingbox(0,0)rectangle(2,1);
                    \draw[thick](0,0)--(0,1)--(2,1)--(2,0);
                    \GRID
                \end{tikzpicture}%
            };%
        \draw(\COLB,0)node{%
                \begin{tikzpicture}[baseline=2.5,xscale=\SCALE,yscale=\SCALE]
                    \useasboundingbox(0,0)rectangle(2,1);
                    \draw[thick](1,0)--(0,0)--(0,1)--(2,1)--(2,0);
                    \GRID
                \end{tikzpicture}%
            };%
        \draw(\COLC,0)node{%
                \begin{tikzpicture}[baseline=2.5,xscale=\SCALE,yscale=\SCALE]
                    \useasboundingbox(0,0)rectangle(2,1);
                    \draw[thick](0,0)--(0,1)--(2,1)--(2,0)--(1,0);
                    \GRID
                \end{tikzpicture}%
            };%
            \draw(\COLD,0)node{%
                \begin{tikzpicture}[baseline=2.5,xscale=\SCALE,yscale=\SCALE]
                    \useasboundingbox(0,0)rectangle(2,1);
                    \draw[thick](0,0)--(0,1)--(2,1)--(2,0);
                    \draw[thick](1,0)--(1,1);
                    \GRID
                \end{tikzpicture}%
            };%
            \draw(2.5,\ROW)node{$\to$};
            \draw(2.5,0)node{$\to$};
            \draw(6.75,4.5)node{bad moves};
            \draw(\COLD,4.5)node{better};
    \end{tikzpicture}
    \caption{Opening a $2$-chain in one of the ``bad'' ways allows
    the opponent to choose between capturing both boxes and then 
    moving again, or replying with the hard-hearted handout.  
    Bisecting the $2$-chain as shown on
the right removes the opponent's second option.}
    \label{FigTwoChains}
\end{figure}

A chain of length~$2$ does not present the same problem,
because if you split it into two single boxes then your opponent 
has no choice but to open the next component.
Two examples appear in figure~\checkedref{FigTwoChains}.
The ``long'' in ``long chain'' restricts attention to 
chains of length $3$ or
more.  In this paper we are restricting attention to 
positions consisting of loops and long chains, so chains of length~$2$
will not appear.  So henceforth we will  write ``chain'' for
long chain, sometimes saying ``long chain'' just for emphasis.

If C always used the hard-hearted handout, then she would 
capture all but~$2$ from each chain and all but $4$ from each loop,
except that she would capture all of the last component.  
If the loops and chains are not too small then
this leads to very lopsided scores.  To first approximation
the game is always lost by whoever has to open the first loop or
long chain.  
This is the connection between Dots~\& Boxes and
Nimstring, and leads to
the \emph{long chain rule} \cite[Ch.~16]{WinningWays}.  
This rule predicts that
the first resp.\ second player will win if there are an even resp.\ odd 
number of long chains in the endgame (assuming an 
odd${ }\times{}$odd board).  So in advanced play, the early game and
midgame are all about trying to create chains, or obstruct their creation,
so that the desired  number (mod~$2$) are formed.  

As in the introduction,
we call the two players the \emph{opener} and \emph{controller}.
Which player is which can change 
during a game. Whoever has no choice but to open
some chain or loop is the opener, and whoever 
replies to a just-opened chain or loop is the controller.
In figure~\checkedref{FigExamples} the author was the opener and his
opponent the controller.  But sometimes
the controller should give up control by taking all of an opened 
chain or loop.  By doing this she
becomes the opener and her opponent becomes the
controller.

For example, if the position consists of $N>1$ chains
of length~$3$, one of which was just opened, 
then a controller who 
keeps control except at the end
will give her opponent $2$ boxes from each chain but the last.  
She scores  $(N-1)+3$ to her opponent's $2(N-1)$, so she loses if $N>4$.
If $N>3$ then
it is better to take all $3$ boxes of the opened
chain.  Being forced to move again,
she is now the opener: she opens a chain and her opponent 
(now the controller) faces a
similar decision.  In one line of optimal play,  the players
take turns giving up control, except that whoever responds in the second-to-last chain
keeps control, giving up $2$ boxes there but getting all~$3$ in the last chain.
If $N$ is even then  the original controller will be the one to do this, and wins
$\frac32N+1$ to $\frac32N-1$.
When $N$ is odd she does not get to do this, but still wins 
$\frac32N+\frac12$ to $\frac32N-\frac12$.

This article is about deciding what to open when you must open something, and deciding
whether to keep control when you have it.

\bigskip
The rest of this section
is more technical.
We have said that 
we will consider only 
Dots~\& Boxes positions $G$ consisting of loops
and long chains, where long means of length${ }\geq3$.  
This is not literally true in two cases.
Immediately after the opener has opened
a component $C$, and whenever the controller 
keeps control, one or two opened components are present.
But we will be able to phrase all our analysis in terms of $G-C$.
So we adopt the convention that unless opened components are 
specifically allowed, Dots~\& Boxes positions consist of
\emph{unopened} loops and long chains.
We also note that
the length of a loop must be even and at least~$4$, 
by the geometry of the grid.

If $G$ is a Dots~\& Boxes position then we define 
\begin{center}
  \begin{tabular}{rl}
      $\size(G)$&number of boxes not yet claimed
    \\
    $\theta(G)$&number of $3$-chains
    \\
    $f(G)$&number of $4$-loops
    \\
      $s(G)$&number of $6$-loops
    \\
      $v(G)$&value---see section~\checkedref{SecIntro} and below
    \\
      $c(G)$&controlled value---see \checkedeqref{EqDefControlledValue} and below
    \\
      $\tb(G)$&terminal bonus---see \checkedeqref{EqTB}
  \end{tabular}
\end{center}
Recall that when discussing $G$ we ignore any boxes captured earlier
in the game, so $\size(G)$ is not the total number of boxes
in the grid, but  the number of
unclaimed boxes.
We call $G$ even or odd according to whether  $\size(G)$ is even or odd.
Usually the symbol
$G$ will represent whatever position is of central interest, so
we will 
use the abbreviations $\theta$, $f$, $s$,
$v$ and $c$ for $\theta(G)$, $f(G)$, $s(G)$, $v(G)$ and $c(G)$.  
We will
still use functional notation for other positions, for example  $\theta(G+3)=\theta+1$.

We recall the meaning of 
the value $v(G)$ from section~\checkedref{SecIntro}:
the opener will lose by $v(G)$ boxes
if both players  maximize the number of boxes they take. 
The trivial example is that $v(G)=\size(G)$ when $G$ consists of a 
single component.  After the opener opens it, the controller
``gives up control'' by claiming all the boxes and ending
the game.
The next example is $v(3+3)=2$.
The opener opens a chain, the controller keeps control (replying
with the hard-hearted handout).  
The opener accepts the $2$ boxes in the handout and 
opens the other chain, all of which the controller claims.
So $v(3+3)=4-2=2$.  

Once the opener has opened a component $C$, the controller
has really only two choices. First, he may keep control and then
play optimally, in which case he will finish the game with
\begin{center}
\begin{tabular}{rcl}
    $(\size(C)-4)+v(G-C)$& if &$C$ is a chain
    \\
    $(\size(C)-8)+v(G-C)$& if &$C$ is a loop
\end{tabular}
\end{center}
more boxes than his opponent.  Second, he may
give up control and then play optimally, in which case 
he will win by
    $\size(C)-v(G-C)$.
The minus sign appears  because when the controller gives up 
control, he enters the position
$G-C$ as the opener, so his opponent is the
one who will score $v(G-C)$.  
We write $v(G;C)$ for the higher of these two margins of victory.
One should think of this as the value of $G$, given that the opener
has just opened~$C$.
Obviously the controller
will choose the higher-scoring option,
leading to
\begin{equation}
    \label{EqvGgivenC}
    v(G;C)=
    \begin{cases}
        (\size(C)-2)+\bigl|v(G-C)-2\bigr|
        &
        \hbox{if $C$ is a chain}
        \\
        (\size(C)-4)+\bigl|v(G-C)-4\bigr|
        &
        \hbox{if $C$ is a loop}
    \end{cases}
\end{equation}
We will use this repeatedly without specific reference.
Another way to express this is that if $v(G-C)=2$ resp.~$4$
and $C$ is a chain resp.\ loop, then keeping and giving up control 
are equally good options.  
If $v(G-C)$ is more than this then
the controller should keep control, and if it is less then he should 
give up control. 

Now consider the opener's perspective,
facing a nonempty position~$G$. He will obviously prefer to open 
whichever component $C$ minimizes $v(G;C)$.  So 
\begin{equation}
    \label{EqValueAsMin}
    v(G)=\min_C v(G;C)
\end{equation}
where $C$ varies over the components of~$G$.  
The previous paragraph shows that each $v(G;C)$
is nonnegative, which proves $v(G)\geq0$.
This justifies our assertion that the opener can never win the endgame.
Also, combining \checkedeqref{EqValueAsMin}
and the previous
paragraph gives a way to evaluate $v(G)$ recursively.  As
an example we work out how to play the position $G=3^n$,
meaning that $G$ 
consists of $n$ many $3$-chains.  The opener has no real choice
about which component to open, so
$v(3^n)=v(3^n;3)=1+|v(3^{n-1})-2|$ whenever $n>0$.  Induction gives
$v(3^n)=0,3,2,1,2,1,2,1,\dots$ when $n=0,1,2,3,4,5,6,7,\dots$.  
So the controller
must
keep control if $n=2$, and must
give up control if either $n=1$ or $n$ is even and larger than~$2$.
In all other cases the choice of keeping or giving up control
makes no difference. 
The simplest rule is to keep control only when $n=2$.

\medskip
Now we can explain the controlled value $c(G)$, defined in 
\checkedeqref{EqDefControlledValue}.
Another useful formula for it is
\begin{equation}
    \label{EqCV}
    c(G) = \sum_i(c_i-4) + \sum_j(l_j-8) + \tb(G)
\end{equation}
where
$c_1,c_2,\dots$ are the lengths of the chains,
$l_1,l_2,\dots$ are the
lengths of the loops, and
the terminal bonus was defined in \checkedeqref{EqTB}.
Ber\-le\-kamp introduced the idea of a controller who follows the 
\emph{control strategy} of always keeping control, except in
a few cases when
he obviously shouldn't.  
In our formulation,
the controlled value is the margin
by which a controller publicly committed to
this strategy will win the endgame $G$ 
against a skilled opener.
The role of the public commitment is to simplify the analysis
of his opponent's strategy.  
As usual, $c(G)<0$ indicates a loss for the controller not a win.

Our formulation of the control strategy is:
keep control until the opener opens the last
component, or opens a loop and only $3$-chains remain.  
In these cases, give up control.
In the second case, play continues after giving up control, 
the
player is now the opener, and only $3$-chains remain.
He  opens them until  he regains control, if he ever
does, 
when he returns to the control
strategy.  

Our  control strategy differs from the one 
in \cite[p.~84]{Berlekamp}
in how the controller plays if he gives
up and then later recovers control.  Although
we don't need it, we remark that this change
doesn't affect the final score.
See lemma~13 in \cite{BC} for the analogue  
of our next result.

\begin{theorem}[Controlled value]
    Suppose $G$ is a Dots \& Boxes position consisting of
    loops and long chains.
    If the controller  follows the
    control strategy, and the opener knows this and plays
    optimally, then the controller
    will win by~$c(G)$.
\end{theorem}

\begin{proof}
    Because the roles of controller and opener can change 
    during play, we write $K$ resp.\ $O$ for the player who
    starts out as the controller resp.\ opener. 
    We write $d(G)$ for the amount by which $K$ will win,
    assuming best play
    by~$O$.  We must show $d(G)=c(G)$.  
    If $C$ is
    a component then we write
    $d(G;C)$ for what the final score would be if $O$ opens~$C$
    and then plays optimally.  Obviously $d(G)=\min_C d(G;C)$.

    We induct on the number of components.  
    The base case is that $G$ has one component (or none).
    Then 
    $c=\size(G)=v$.  
    For the inductive step,
    suppose first that $G$ does not consist of a loop and
    (one or more) $3$-chains.
    After $O$ opens  component~$C$,
    $K$ will keep control, 
    so
    \begin{align*}
        d(G;C)&{}=\size(G-C)-
    \begin{pmatrix}
        \hbox{$4$ if $C$ is a chain}
        \\
        \hbox{$8$ if $C$ is a \rlap{loop}\phantom{chain}}
    \end{pmatrix}
        +d(G-C)
        \\
        &{}=c(G)-\tb(G)+\tb(G-C)
    \end{align*}
    We have used the inductive hypothesis
    $d(G-C)=c(G-C)$ and applied  \checkedeqref{EqDefControlledValue} to both
    $G-C$ and~$G$.   This shows that $O$ should open the
    component $C$ that minimizes $\tb(G-C)$.  
    If $C$ is a loop then removing it leaves the terminal bonus
    invariant.  (The only way it could change is if it were the last
    loop and only $3$-chains remain, but we have set that
    case aside.)  
    If $C$ is a chain then removing it increases
    the terminal bonus or leaves it the same.  
    Therefore, if there is a loop 
    then $O$ should choose it as $C$, and
    $$d(G)=d(G;C)=c(G)-\tb(G)+\tb(G-C)=c(G)$$
    If there is no loop then the removal of any chain leaves the terminal
    bonus invariant, and the same calculation shows $d(G)=c(G)$.

    It remains to consider the case that $G$ consists of a loop~$L$
    and  $n\geq1$ many $3$-chains.  We write $l$ for the length
    of the loop.
    Our argument for $d(G;C)$ still applies if $C=3$, namely
    $$
    d(G;3)=c(G)-\tb(G)+\tb(G-3)
    =
    \begin{cases}
        l-1&\hbox{if $n=1$}
        \\
        l-n-2&\hbox{if $n>1$}
    \end{cases}
    $$
    The key point is to compute $d(G;L)$.  If $O$ opens $L$, then $K$
    will give up control, and then $O$ may keep control for however many 
    turns~$m\leq n$
    he desires.  In fact $O$ will choose $m<n$ because giving away two
    boxes of the last component would be silly.
    After $O$ gives up control, $K$ will keep it until 
    the last chain.  This leads to 
    $$
    \hbox{final score}=
    \begin{cases}
        l+n-4&\hbox{if $m=n-1$}
        \\
        l+2m-n-2&\hbox{if $0\leq m<n-1$}
    \end{cases}
    $$
    If $n>1$ then all these values are${}\geq d(G;3)$, so
    $O$ should open the $3$-chain and $d(G)=d(G;3)=c(G)$.  If $n=1$
    then $m$ must be~$0=n-1$, so $d(G;L)=l-3$.  
    Since this is less than $v(G;3)=l-1$, $O$ should open the loop. And 
    $d(G)=d(G;L)=l-3=c(G)$.
\end{proof}

We close this section with some relations between the 
various quantities we have introduced.  

\begin{lemma}
    \label{LemTrivRlns}
    Suppose a Dots~\& Boxes position $G$ consists of loops and long chains.
Then
\begin{enumerate}
    \item
        \label{ItemCatmostV}
        $c(G)\leq v(G);$
    \item
        \label{ItemVCSizemod2}
        $\size(G)\cong c(G)\cong v(G)\cong v(G;C)$ mod $2$,
        for every component $C$ of $G$;
    \item
        \label{ItemCSizeMod4}
    $\size(G)\cong c(G)$ mod $4$ if $G$ has no $3$-chains.
    \qed
\end{enumerate}
\end{lemma}

We also mention two more relations that we will prove
later.  
First,
adding the hypothesis that $G$ is even to part \checkedeqref{ItemCSizeMod4}
strengthens its conclusion to 
$v(G)\cong\size(G)\cong c(G)$ mod~$4$.  (See theorem~\ref{ThmValuesExplicit}.)
Second, our theorem~\checkedref{ThmCAtLeast2}
contains Berlekamp's result 
$c(G)\geq2\implies c(G)=v(G)$.  

\begin{proof}
    \checkedeqref{ItemCatmostV} The controller may guarantee a final score of
    at least $c(G)$ by publicly committing to the control strategy.

    \checkedeqref{ItemVCSizemod2} Both $c(G)$ and $v(G)$ have the same
    parity as $\size(G)$ because they are margins of victory under
    certain lines of play.  For the last congruence, it follows from
    \checkedeqref{EqvGgivenC} that $v(G;C)$ has the same or different parity
    as $v(G-C)$
    according to whether $\size(C)$ is even or odd.  
    Also,  $\size(G)$ has the same or different parity as
    $\size(G-C)$ under the same conditions.  Since
    $\size(G-C)$ and $v(G-C)$ have the same parity, it follows that
    $v(G;C)\cong\size(G)$ mod~$2$.

    \checkedeqref{ItemCSizeMod4} When no $3$-chains are present the terminal
    bonus is divisible by~$4$.  So our claim follows from the 
    definition \checkedeqref{EqDefControlledValue} of $c(G)$.
\end{proof}

\section{Opener strategy}
\label{SecOpener}

\noindent
Our goal in this section is to prove
theorem~\checkedref{ThmOpener}, under the
standing hypothesis is that $G$
is a Dots~\& Boxes position consisting of
unopened loops and long chains.
Although our proof is logically independent of the
work of Buzzard and Ciere \cite{BC},
and organized very differently, we would
not have been able to formulate theorem~\checkedref{ThmOpener}
or the lemmas below
without reference to it.
We regret that we do not use 
their lovely man-in-the-middle and chain- and loop-amalgamation
arguments.

\begin{lemma}[Tiny positions]
  \label{LemTinyPositions}
  Suppose $G$ has one component, or consists of a $3$-chain and a loop.  Then $v(G)=c(G)$.
  In the second case, opening the loop is optimal.
\end{lemma}

\begin{proof}
  The first case is obvious.  For the second case one examines the four possible lines of play.
  (Opening the $3$-chain is not optimal.)
\end{proof}

The next theorem summarizes the key points of the lemma after it.
After establishing it
we will use the implication $(c\geq2)\implies(v=c)$ many
times without specific
reference.  
The lemma and its proof are essentially the same as
lemma~15 from \cite{BC}, that one can ``fly the plane
without crashing'' in the sense of \cite{BS}.

\begin{theorem}[Large controlled values]
  \label{ThmCAtLeast2}
  Suppose $c(G)\geq2$. Then $v(G)=c(G)$ and the following gives an optimal move:
  \begin{enumerate}
    \item
      Open a $3$-chain  (if $G$ has one and at least one other chain).
    \item
      Otherwise, open a shortest loop (if $G$ has a loop).
    \item
      Otherwise, open a shortest chain.
      \hfill$\Box$
  \end{enumerate}
\end{theorem}

\begin{lemma}[Large controlled values---details]
  \label{LemCAtLeast2}
  Suppose $c(G)\geq2$.  
  \begin{enumerate}
    \item
      \label{ItemCAtLeast2OptimalComponent}
      If $C$ is any component satisfying
      \begin{equation*}
        \tb(G-C)=\tb(G)
        \hbox{ \ and \ }
        c(G-C)\geq
        \begin{cases}
          2&\hbox{if $C$ is a cycle}
          \\
          4&\hbox{if $C$ is a loop}
        \end{cases}
      \end{equation*}
      then opening $C$ is optimal.  
    \item
      \label{ItemCAtLeast2MoveIn3}
      If $G$ has a $3$-chain, $4$-loop or $6$-loop, whose removal 
      does not alter the terminal bonus, then
      opening it is optimal.
    \item
      \label{ItemCAtLeast2MoveInLoop}
      Suppose $G$ has no $3$-chains, or exactly one $3$-chain and no other chains.
      If  $G$ has a loop 
      then opening a
      shortest loop is optimal.
    \item
      \label{ItemCAtLeast2MoveInChain}
      If $G$ has no loops, then opening a shortest chain is optimal.
    \item
      \label{ItemCAtLeast2Value}
      $v(G)=c(G)$.
  \end{enumerate}
\end{lemma}

\begin{proof}
  We  induct on the number of components.  If $G$ has one component then 
  \checkedeqref{ItemCAtLeast2OptimalComponent}--\checkedeqref{ItemCAtLeast2MoveIn3} are vacuous and
  \checkedeqref{ItemCAtLeast2MoveInLoop}--\checkedeqref{ItemCAtLeast2Value} are trivial.
  So suppose $G$ has at least
  two components.  

  \checkedeqref{ItemCAtLeast2OptimalComponent}  
  Write $l$ for the length of $C$, and 
  assume first that $C$ is a loop.  We have
  \begin{equation*}
    v(G;C)=(l-4)+|v(G-C)-4|
  \end{equation*}
  Since $c(G-C)\geq4$ by the hypothesis of 
  \checkedeqref{ItemCAtLeast2OptimalComponent}, induction gives $v(G-C)=c(G-C)\geq4$. So
  we have
  \begin{equation*}
    v(G;C)=(l-4)+c(G-C)-4= (l-8)+c(G-C)
  \end{equation*}
  Since $G$ and $G-C$ have the same terminal bonus, $c(G-C)=c(G)-(l-8)$. Plugging this
    in gives $v(G;C)=c(G)$.  Together with 
    lemma~\ref{LemTrivRlns}
    this gives
    $$
    c(G)\leq v(G)\leq v(G;C)=c(G)
    $$ 
    Therefore $v(G;C)=v(G)$, which 
    proves optimality, and $v(G)=c(G)$.
  If $C$ is a chain then the argument is the same with all $4$'s replaced by $2$'s 
  and all $8$'s by~$4$'s.
  
  \checkedeqref{ItemCAtLeast2MoveIn3} First suppose $G$ has a $3$-chain with $\tb(G-3)=\tb(G)$.  Then
  $c(G-3)=1+c(G)\geq3$.  So \checkedeqref{ItemCAtLeast2OptimalComponent} applies to the $3$-chain.
  And similarly for a $4$-loop or $6$-loop.

  \checkedeqref{ItemCAtLeast2MoveInLoop} 
  By hypothesis,  $G$ consists of  either at least one
  loop and a $3$-chain, or at least one loop and possibly some chains 
  of length${ }\geq4$.
  Let $C$ be a shortest loop.  
  In the special case  $G=C+3$, opening
  $C$ is optimal by lemma~\checkedref{LemTinyPositions}.   So suppose otherwise: $G$ has a second loop, or
  a chain of length${ }\geq4$.  In either case the removal of $C$ leaves the terminal bonus
  invariant.  
  If $C=4_\ell$ or $6_\ell$ then we are done by \checkedeqref{ItemCAtLeast2MoveIn3}.
  There are two remaining cases.
  First, $G$ consists of a $3$-chain and two or more
  loops of length${ }\geq8$. Second, $G$ consists of at least one  loop of length${ }\geq8$ and 
  possibly some chains 
  of length${ }\geq4$.  
  In each case we write out  formula~\checkedeqref{EqCV} for
  $c(G-C)$ and use the absence of $4$- and $6$-loops.
  The results in the two cases are
  \begin{align*}
    c(G-C) & { }=(3-4)+(\hbox{nonnegative terms})+6\geq5
    \\
    \llap{and}\quad
    c(G-C) & { }= (\hbox{nonnegative terms})+(\hbox{$4$ or $8$}) \geq4
  \end{align*}
  So \checkedeqref{ItemCAtLeast2OptimalComponent} shows that opening $C$ is optimal.
  
  \checkedeqref{ItemCAtLeast2MoveInChain} $G$ consists of at least $2$ chains, so removing any one of them
  leaves the terminal bonus unchanged.  Let $C$ be  shortest possible.  If it is a $3$-chain then we
  appeal to \checkedeqref{ItemCAtLeast2MoveIn3}.  Otherwise we mimic the argument at the end of the proof of \checkedeqref{ItemCAtLeast2MoveInLoop}:
  \begin{equation*}
  c(G-C) = 4 + (\hbox{nonnegative terms})\geq4
  \end{equation*}
  Since the right side is at least~$2$, it follows from 
  \checkedeqref{ItemCAtLeast2OptimalComponent} that opening $C$ is optimal.

  \checkedeqref{ItemCAtLeast2Value} One of \checkedeqref{ItemCAtLeast2MoveIn3}--\checkedeqref{ItemCAtLeast2MoveInChain}
  applies to $G$.  Their proofs show that either $G$ consists of a $3$-chain and a loop, or
  else $G$ has a component $C$ satisfying~\checkedeqref{ItemCAtLeast2OptimalComponent}.  In the first
  case we appeal to lemma~\checkedref{LemTinyPositions} for the equality $v=c$.  In the second case
  we observed $v=c$ during the proof of 
  \eqref{ItemCAtLeast2OptimalComponent}.
\end{proof}

\begin{lemma}[Example of $6_\ell$-optimality]
  \label{Lem6Loptimality}
  Suppose $G\neq\emptyset$ has no $3$-chains or $4$-loops, and $c(G)\leq4$.  Then 
      \begin{equation*}
        v(G) =
        \begin{cases}
          \phantom{\hbox{\rm(whichever of $0,4$)}}
          3&\hbox{if $\size(G)$ is odd}
          \\
          \hbox{\rm(whichever of $2,4$)${ }\cong\size(G)$ mod~$4$}&\hbox{otherwise}
        \end{cases}
      \end{equation*}
      Furthermore, if $c(G)<2$ then 
      $G$ has at least two $6$-loops and  opening one
      of them is optimal.
\end{lemma}

\begin{proof}
  We write $w(\cdot)$ for the function on positions defined by the formula.
  First we establish the cases $c=2,3,4$.  
    By the absence of $3$-chains, lemma~\ref{LemTrivRlns}\eqref{ItemCSizeMod4}
    gives $\size(G)\cong c$ mod~$4$. So the definition of $w$
  gives  $w(G)=c$.  And theorem~\checkedref{ThmCAtLeast2} gives
  $v=c$, completing the proof when $c\geq2$.

  So 
  we may suppose $2>c$, which we write out using \checkedeqref{EqCV}
  \begin{equation*}
    2>c=-\theta-4f-2s+(\hbox{nonnegative terms})+\tb(G)
  \end{equation*}
  By $\tb(G)\geq4$ and $\theta=f=0$ we get $s\geq2$.  
  Since there are at least two loops,
  the removal of any loop leaves the
  terminal bonus unchanged.  
  In particular, $c(G-6_\ell)=c+2$. Also, 
 $G$ has a third
    component because $c(G)<2$ and 
    $c(6_\ell+6_\ell)=4$.  Having made these preparations,
  we now induct on the number of components.  
  
  First we claim $v(G-6_\ell)=w(G-6_\ell)$.  
  If $c(G-6_\ell)\geq2$ then $c(G-6_\ell)=2$ or~$3$, which are cases already proven.
  This  includes the base case that $G$ has $3$ components,
  because 
then $G-C=6_\ell+C'$ with $C'\neq3,4_\ell$, which forces
    $c(G-C)\geq2$.
  On the other hand, if $c(G-6_\ell)<2$ then induction gives $v(G-6_\ell)=w(G-6_\ell)$.

  Next we claim $v(G;6_\ell)=w(G)$.  The previous paragraph gives us
  \begin{equation}
      \label{EqValueRel6L}
    v(G;6_\ell)=2+|w(G-6_\ell)-4|=
    \begin{cases}
      3&\hbox{if $\size(G-6_\ell)$ is odd}
      \\
      2&\hbox{if $\size(G-6_\ell)\cong4$ mod~$4$}
      \\
      4&\hbox{if $\size(G-6_\ell)\cong2$ mod~$4$}
    \end{cases}
  \end{equation}
  Because the sizes of $G$ and $G-6_\ell$ differ by $2$ mod~$4$, the right side is 
  $w(G)$, proving the claim.  Once we prove that opening a $6$-loop is optimal
  it will follow that $v=v(G;6_\ell)=w(G)$ and the induction will be complete.
  
  Suppose that some component $C\neq6_\ell$ 
  has $v(G;C)<v(G;6_\ell)$.  
  We have 
  \begin{equation}
      \label{EqInequalityArgument}
    2\leq v(G;C) \mathop{\cong}_{\mathrm{mod}\,2}
    v(G;6_\ell)\leq4
  \end{equation}
  The first inequality comes from
    \eqref{EqvGgivenC}
    because $C$ is a chain of length${ }\geq4$
      or a loop of length${ }\geq8$.  
      The congruence is 
      Lemma~\checkedref{LemTrivRlns}\checkedeqref{ItemVCSizemod2},
  and we saw $v(G;6_\ell)\leq4$ in the previous paragraph.
This forces 
\begin{equation}
    \label{EqInequalityArgument2}
v(G;C)=2
\quad\hbox{and}\quad 
v(G;6_\ell)=4
.
\end{equation}
    Together with \eqref{EqvGgivenC},
    the first shows that $C$ is a $4$-chain with $v(G-4)=2$.
    Together with \checkedeqref{EqValueRel6L}, the second
    shows that $\size(G)\cong4$ mod~$4$.

    This is impossible because $v(G-4)\cong4$ mod~$4$:  by induction
    if $c(G-4)<2$,  or by 
    \begin{equation*}
      v(G-4)=c(G-4)\cong\size(G-4)\cong4\hbox{ mod~$4$}
    \end{equation*}
    if $c(G-4)\geq2$.  (The first congruence is lemma~\checkedref{LemTrivRlns}.)
\end{proof}

We will use the argument for \checkedeqref{EqInequalityArgument}
and \checkedeqref{EqInequalityArgument2}
several more times, without giving the details each time.

\begin{lemma}[First example of $3$-optimality]
    \label{Lem3Optimality1}
  Suppose $c(G)<2$ and that $G$ has no $4$-loops, but does have a $3$-chain.  
  Then opening a $3$-chain is optimal, and
  \begin{equation*}
    v(G)=
    \begin{cases}
      2&\hbox{if $G$ is even}
      \\
      3&\hbox{if $\theta=1$ and $\size(G)\cong3$ mod~$4$}
      \\
      1&\hbox{otherwise}
    \end{cases}
  \end{equation*}
\end{lemma}

\begin{proof}
  We write $w(\cdot)$ for this function on positions,
    and prove $v=c$ by induction
    on $\theta$.  In both the base case
  and the inductive step the strategy is to show $v(G;3)=w(G)$.  Given this,
  the optimality of a $3$-chain follows 
  from
  \begin{equation*}
    2\leq v(G;C)\mathop{\cong}_{\hbox{\scriptsize mod~$2$}} v(G;3)=w(G)\leq3 
  \end{equation*}
  for every component $C$  other than a $3$-chain.
  
  Now for the induction.  
  By the hypothesis $c<2$, $G$ does not consist of a single
    $3$-chain, so there is another component.
    This implies that the terminal bonus rises by at
    most~$2$ if a $3$-chain is removed.
    In particular, 
    $c(G-3)\leq c+3\leq4$.  If
    $\theta=1$ then  $v(G-3)$ is given by lemma~\checkedref{Lem6Loptimality}.
  This yields
  \begin{equation*}
  v(G;3)=1+|v(G-3)-2|
  =\begin{cases}
      2&\hbox{if $\size(G-3)$ is odd}
      \\
      3&\hbox{if $\size(G-3)\cong4$ mod~$4$}
      \\
      1&\hbox{if $\size(G-3)\cong2$ mod~$4$}
  \end{cases}
  \end{equation*}
  This is visibly equal to $w(G)$.  Now suppose $\theta>1$.  
  We must show that $v(G;3)$ is equal to whichever of $1$ and~$2$ has
  the same parity as~$G$.  
  It is enough to show $v(G;3)\in\{1,2\}$.  In turn,
    this will follow once we prove
  $v(G-3)\in\{1,2,3\}$.
  We observe
  $c(G-3)=c+1\leq2$.  In the case $c(G-3)=2$ 
  we have $v(G-3)=c(G-3)=2$.  
  In the case $c(G-3)<2$ we have $v(G-3)\in\{1,2,3\}$ by induction.
  This finishes  the proof.
\end{proof}

\begin{lemma}[$3$'s and $4_\ell$'s]
  \label{Lem3sAnd4Ls}
  Suppose $G$ consists of  $3$-chains and $4$-loops.  Then
  \begin{equation*}
    v(G) =
  \begin{cases}
    \hbox{\rm(whichever of $0,4$)${ }\cong\size(G)$ mod~$8$}&\hbox{if $\theta=0$}
    \\
    \phantom{\hbox{\rm(whichever of $0,4$)}}
      3&\hbox{if $G=3+4_\ell^{\even}$}
    \\
    \hbox{\rm(whichever of $1,2$)${ }\cong\size(G)$ mod~$2$}&\hbox{otherwise}
  \end{cases}
  \end{equation*}
  Furthermore, 
  \begin{enumerate}
      \item
           Opening a $4$-loop is optimal
          if and only if $\theta$ is even or~$1$.
      \item
           Opening a $3$-chain is optimal
          if and only if $\theta\geq2$
          or $f$ is even.
  \end{enumerate}
\end{lemma}

\begin{proof}
  Each entry in the following table is the smaller of
    \begin{align*}
        v(G;4_\ell)&{}=\bigl|(\hbox{the entry to the left})-4\bigr|
        \\
        \llap{and\qquad}
        v(G;3)&=1+\bigl|(\hbox{the entry above})-2\bigr|.
    \end{align*}
  Except: in the left column there is no $v(G;4_\ell)$ and in 
  the top row there is no $v(G;3)$.
  The top left entry is $v(\emptyset)=0$.  
It follows by induction that the table gives $v(G)$,
  which justifies all our claims.
  \begin{center}
    \begin{tabular}{lrcccccccl}
      &$f=$&$0$&$1$&$2$&$3$&$4$&$5$&$6$&$\cdots$
      \\
      \noalign{\vskip1pt\hrule\vskip2pt}
      $\theta=0$ case:&$v(G)=$&
      0&4&0&4&0&4&0
      \\
      $\theta=1$ case:&$v(G)=$&
      3&1&3&1&3&1&3
      \\
      $\theta=2$ case:&$v(G)=$&
      2&2&2&2&2&2&2
      \\
      $\theta=3$ case:&$v(G)=$&
      1&1&1&1&1&1&1
    \end{tabular}
  \end{center}
  Starting with $\theta=2$, the rows 
  alternate between all $2$'s and all $1$'s.
\end{proof}

\begin{lemma}[First example of $4_\ell$-optimality]
  \label{Lem4Optimality1}
  Suppose $G$ has a $4$-loop and $c(G)\geq-2$.  Then
  \begin{enumerate}
    \item
      \label{Lem4Optimality1value}
      $v(G)=|c(G)|$ unless  $G=3+3+4_\ell$, in which case $v(G)=2$.
    \item
      \label{Lem4Optimality1move}
      Opening a $4$-loop is optimal unless  $G=3+3+3+4_\ell$.
  \end{enumerate}
\end{lemma}

\begin{proof}
  First we treat the case that removing a $4$-loop changes the terminal bonus.  This can
  only happen when $G$ consists of a $4$-loop and zero
    or more $3$-chains.  The condition
  $c\geq-2$ shows that the number of $3$-chains is $0,1,2,3$ or~$4$.  In these cases
  we have $c=4,1,0,-1$ or~$-2$, and   lemma~\checkedref{Lem3sAnd4Ls} gives 
  $v=4,1,2,1$ or $2$ respectively.  As claimed, $v=|c|$ unless
  $G=3+3+4_\ell$.  
  For the optimality of opening a $4$-loop
    when 
    $G=$(no, one, two or four $3$-chains)
  we refer to  lemma~\checkedref{Lem3sAnd4Ls}.

  Now suppose removing a $4$-loop leaves the terminal bonus invariant.  
  So $c(G-4_\ell)=c+4$.
  We first prove $v(G;4_\ell)=|c|$.
  By $c\geq-2$ we have $c(G-4_\ell)\geq2$, so $v(G-4_\ell)=c(G-4_\ell)$.
  If $c=-2$ resp.~$-1$, then $c(G-4_\ell)=2$ resp.~$3$, so
  $v(G-4_\ell)=2$ resp.~$3$, so 
  $v(G;4_\ell)=2$ resp.~$1$, which equals $|c|$.
  And if $c\geq0$ then
  \begin{equation*}
    v(G;4_\ell)
    =|c(G-4_\ell)-4|
    =(c+4)-4=c=|c|.
  \end{equation*}
  This completes the proof that $v(G;4_\ell)=|c|$.

  If $c\geq2$ then the optimality of opening a $4$-loop 
  is lemma~\checkedref{LemCAtLeast2}\checkedeqref{ItemCAtLeast2MoveIn3}.  Otherwise, 
  $v(G;4_\ell)=|c|=0$, $1$ or~$2$
  by the previous paragraph. 
  If $C$ is a component other than a $4$-loop, then 
  \begin{equation*}
    1\leq v(G;C)\mathop{\cong}_{\mathrm{mod}\,2} v(G;4_\ell)\leq2
  \end{equation*}
  So $v(G;C)$  cannot be less than
  $v(G;4_\ell)$.
  This proves the optimality of a $4$-loop,
  hence $v(G)=v(G;4_\ell)=|c|$.
\end{proof}

\begin{lemma}[Second example of $4_\ell$-optimality]
    \label{Lem4Optimality2}
  Suppose $c<2$ and that $G\neq\emptyset$ has no $3$-chains.
  If $G$ has a $4$-loop then opening it is optimal.
  Regardless of whether $G$ has a $4$-loop, if $c+4f\geq2$ then
  \begin{align}
    v&{ }= 0,1,2,3,4
    \hbox{\ \ in the cases\ \ $c\cong0,\pm1,\pm2,\pm3,4$ mod~$8$}
    \label{EqX}
    \\
    \noalign{\parindent=0pt while if $c+4f<2$ then}
    v&{ } =
    \begin{cases}
      \phantom{\hbox{\rm(whichever of $0,4$)}}
      2&\hbox{if $c\cong2$ mod~$4$}
      \\
      \hbox{otherwise:}
      \\
      \hbox{\rm(whichever of $0,1$)${ }\cong c$ mod~$2$}&\hbox{if $f$ is odd}
      \\
      \hbox{\rm(whichever of $3,4$)${ }\cong c$ mod~$2$}&\hbox{if $f$ is even}
    \end{cases}
    \label{EqY}
  \end{align}
\end{lemma}

\begin{remark}
    In 
    \checkedeqref{EqY} we could replace every occurrence of $c$ by
    $\size(G)$, because the absence of $3$-chains implies
    $\size(G)\cong c(G)$ mod~$4$.
\end{remark}

\begin{proof}
    We begin with three special cases.
    First, suppose 
    that $G$ is a union of $4$-loops.  Then $c+4f=-4f+4f+8\geq2$, so
    we are asserting that $v(G)$ is given by \checkedeqref{EqX}.
    This is justified by lemma~\checkedref{Lem3sAnd4Ls}.

  Second, suppose $G$ has no $4$-loops.  Then $c+4f=c<2$, so we are asserting that $v(G)$ is given by \checkedeqref{EqY}.  This is justified by 
  lemma~\checkedref{Lem6Loptimality}.  

  Third, suppose $G$ has a $4$-loop and that $c=-2,-1,0$, or~$1$.  
  Then $c+4f\geq2$, so we are asserting that $v(G)$ is given by \checkedeqref{EqX}, namely $v(G)=2,1,0$,
  or~$1$ respectively.
  This is justified by 
  lemma~\checkedref{Lem4Optimality1}.  
  
  For the general case we use induction on the number of
  components.  So suppose that every position with fewer components
  than $G$, that satisfies the hypotheses of the lemma, also 
  satisfies its conclusions.  Regarding the cases already treated
  as base cases, we may suppose that $c(G)<-2$ and that
  $G$ has a $4$-loop and a component
  other than a $4$-loop. hence a longer loop or a chain of
  length${ }\geq4$.  
  The presence of this extra component shows 
  that  removing a $4$-loop from $G$  
  does not change the terminal bonus.  In particular, 
  $c(G-4_\ell)=c+4<-2+4=2$.  This and  the minimality of $G$ show 
  that $v(G-4_\ell)$ is given by whichever of \checkedeqref{EqX} and \checkedeqref{EqY}
  applies to $G-4_\ell$.
  To figure out which one  applies we observe
  \begin{equation}
      \label{EqKeepSameCondition}
    c(G-4_\ell)+4f(G-4_\ell)=c(G)+4f(G)
  \end{equation}
  This shows that whichever of \checkedeqref{EqX} and \checkedeqref{EqY} 
  claims to describe $v(G)$ does indeed describe $v(G-4_\ell)$.

  Our next step is to prove that $v(G;4_\ell)$ is 
  equal to what the lemma claims is
  $v(G)$.  That is, writing 
  $w(\cdot)$ for the function on positions given by \checkedeqref{EqX}
  and \checkedeqref{EqY}, we will prove $v(G;4_\ell)=w(G)$.
  Regardless of which of \checkedeqref{EqX} and \checkedeqref{EqY} applies to
  $G-4_\ell$, we have $v(G-4_\ell)=w(G-4_\ell)\in\{0,1,2,3,4\}$.
  It follows that
  $v(G;4_\ell)=4-v(G-4_\ell)$.  It is also easy to see that
  $w(G)=4-w(G-4_\ell)$: for \checkedeqref{EqX} this uses the fact that
  $c(G)$ and $c(G-4_\ell)$ differ by $4$ mod~$8$, while for 
  \checkedeqref{EqY} this uses 
  the fact that $f(G)$ and $f(G-4_\ell)$ have different
  parities.
  It follows that $v(G;4_\ell)=w(G)$.

  Now we prove the optimality of a $4$-loop.
  Suppose for a contradiction
  that $C$ is a component of $G$
  with $v(G;C)<v(G;4_\ell)$.  
  By
  \begin{equation*}
    2\leq v(G;C)\mathop{\cong}_{\mathrm{mod}\,2}v(G;4_\ell)\leq4
  \end{equation*}
  we must have
  $v(G;C)=2$ and $w(G)=v(G;4_\ell)=4$.  
  The first of these 
  forces $C$ to be a $4$-chain with $v(G-4)=2$, or a $6$-loop 
  with $v(G-6_\ell)=4$.
  Removing a $6_\ell$ leaves the terminal bonus invariant,
  while removing a $4$-chain increases it by $0$ or~$4$.
  In either case we have
  \begin{align*}
    c(G-C)&{ }\leq c(G)+4< -2+4=2
  \end{align*}
  By induction, the lemma describes $v(G-C)$.

  If $C$ is a $4$-chain then $v(G-4)=2$ forces $c(G-4)\cong2$ mod~$4$,
  regardless of whether \checkedeqref{EqX} or \checkedeqref{EqY} applies to $G-4$.
  But then the absence of $3$-chains shows
  \begin{equation*}
    c(G)\cong\size(G)\cong\size(G-4)\cong c(G-4)\cong2\hbox{\ \ mod~$4$}
  \end{equation*}
  This gives 
  $w(G)=2$, regardless of whether \checkedeqref{EqX} or
  \checkedeqref{EqY} applies to $G$.  
This contradicts our early conclusion that $w(G)=4$.
  If $C$ is a $6$-loop then 
  $v(G-6_\ell)=4$ forces $c(G-6_\ell)\cong0$ mod~$4$, regardless
  of whether \checkedeqref{EqX} or \checkedeqref{EqY} applies to $G-6_\ell$.  
  This forces $c(G)\cong2$ mod~$4$, leading to the same contradiction.

  We have proven the optimality of opening a $4$-loop.
  So $v(G)=v(G;4_\ell)=w(G)$ follows and the proof is complete.
\end{proof}

The next two lemmas, taken together,
are an analogue of sections 11--13 of \cite{BC}.
But the formulations and arguments are quite different.

\begin{lemma}[Third example of $4_\ell$-optimality]
    \label{Lem4Optimality3}
  Suppose $G$ has  exactly one $3$-chain, and $\size(G)\cong3$ mod~$4$ and $c<2$.
  If $G$ has a $4$-loop then opening it is optimal.  
  Regardless of whether $G$ has a $4$-loop, if $c+4f\geq2$ then
  \begin{align}
    v(G)&{ }=
    \hbox{$1$ or~$3$ according to whether $c(G)\cong\pm1$ or~$\pm3$ mod~$8$}
    \label{EqA}
    \\
    \noalign{\parindent=0pt while if $c+4f<2$ then}
    v(G)&{ }=
      \hbox{$1$ or~$3$ according to whether $f$ is odd or even.}
    \label{EqB}
  \end{align}
\end{lemma}

\begin{proof}
    This is similar in structure to the previous proof.
    Note that $c$ is odd because $G$ is.
    We begin with three special cases.
    First, if $G$ has no $4$-loops then 
    $c+4f<2$, so the lemma claims $v=3$.  This is justified by
     lemma~\checkedref{Lem3Optimality1}.  Second,
    if $G$ has a $4$-loop and $c(G)=\pm1$ then 
    $c+4f\geq2$, so the lemma claims that $v=1$ and opening
    a $4$-loop is optimal.  This is justified
    by lemma~\checkedref{Lem4Optimality1}.  
    (The exception $3+3+3+4_\ell$ in that lemma
    is irrelevant because $\theta=1$.)
    Third, suppose that $G$ consists 
    of a $3$-chain and $f>0$ many $4$-loops.  Then  
    $c+4f=5$, so the lemma asserts that $v(G)=1$ or~$3$ 
    according to whether $c\cong\pm1$ or~$\pm3$ mod~$8$.  
    The cases $c\cong -1$ or~$3$ mod~$8$ do not occur, and
    the cases $c\cong 1$ or $-3$ mod~$8$ are 
    equivalent to $f$ being odd or even, respectively.
    So we may quote
    lemma~\checkedref{Lem3sAnd4Ls}.

    For the general case we induct on the number of components.
    So suppose every position with fewer components than~$G$,
    that satisfies the hypotheses of the lemma, also satisfies
    its conclusions.  Regarding the cases already treated as base
    cases, we may suppose that  $c<-2$, and that
  $G$ contains a $4$-loop and also a component that is neither
  a $3$-chain nor a $4$-loop.  In particular, removing
  a $4$-loop does not alter the terminal bonus.  So
  $c(G-4_\ell)=4+c<2$.  By induction, the lemma describes
  $v(G-4_\ell)$. 
    In particular, $v(G-4_\ell)=1$ or~$3$.  These lead to
    $v(G;4_\ell)=3$ or~$1$ respectively.

  Next we claim $v(G;3)=3$.  
  We have $c(G-3)\leq c+4<2$, so lemma~\checkedref{Lem4Optimality2}
  computes $v(G-3)$.  
  Which case of that lemma applies 
  depends on the value of $c(G-3)+4f(G-3)$.  But both
  cases give 
    $$
    v(G-3)
    \cong \size(G-3)
    \cong\size(G)-3
    \cong0\hbox{ mod }4
    $$ 
  So $v(G;3)=1+|\hbox{($0$ or $4$)}-2|=3$.
  This proves the optimality of opening a $4$-loop, because any
  component $C\neq3,4_\ell$ has $v(G;C)$ odd and at least~$2$.
  
  All that remains is to justify the stated value of $v(G)$.  
  We know $v(G)=v(G;4_\ell)=4-v(G-4_\ell)$.  On the other
  hand, the equality \checkedeqref{EqKeepSameCondition} from the previous
  proof also holds here, so that whichever of \checkedeqref{EqA} and
  \checkedeqref{EqB} purports to describe $v(G)$ does indeed describe
  $v(G-4_\ell)$.  Examining these two formulas shows that
  $4-v(G-4_\ell)$ is the claimed value of $v(G)$.  For
      \checkedeqref{EqA} this uses the fact that $c(G)$ and $c(G-4_\ell)$
      differ by~$4$, while for \checkedeqref{EqB} this uses the fact
  that $f(G)$ and $f(G-4_\ell)$ have different parities.
\end{proof}

\begin{lemma}[Second example of $3$-optimality]
    \label{Lem3Optimality2}
  Suppose $c<-1$ and that either
  \begin{enumerate}
      \item
          $\theta\geq2$, or
      \item
          $\theta=1$ and $\size(G)\not\cong3$ mod~$4$.  
  \end{enumerate}
  Then opening a $3$-chain
  is optimal, and $v=1$ or~$2$ according to whether $G$ is odd or even.
\end{lemma}

\begin{proof}
    We begin by treating three special cases.
    First, if $G$ has no $4$-loops  then we quote lemma~\checkedref{Lem3Optimality1}.
    Second,  if $G$ consists of $4$-loops and 
    $3$-chains  then we quote lemma~\checkedref{Lem3sAnd4Ls}. 

    Third, suppose that $G$ consists
    of a $3$-chain and some loops. Then $G$ is odd, so our hypothesis
    on $\size(G)$ forces $\size(G)\cong1$  mod~$4$.   
    Also $c(G-3)=3+c(G)<2$, so lemma~\checkedref{Lem4Optimality2}
    describes $v(G-3)$.  
    Because $c(G-3)\cong\size(G-3)\cong2$ mod~$4$, that lemma
    gives $v(G-3)=2$, independently of whether \checkedeqref{EqX}
    or \checkedeqref{EqY} applies to $G-3$.   From $v(G-3)=2$ we get
    $v(G;3)=1$.  Since $v(G)$ is odd and bounded above
    by $v(G;3)$, 
    the $3$-chain is optimal and $v(G)=1$.  

    We have reduced to the case that
    $G$ has at least two chains, that it has a $4$-loop, and that it has
    a component that is neither a $3$-chain nor a $4$-loop.  
    It follows that 
    the terminal bonus is invariant under the removal of
    either a $3$-chain or a $4$-loop.
    In particular,
    $c(G-3)=1+c<0$ and 
    $c(G-4_\ell)=4+c<3$.  

    As a fourth special case, suppose $c=-2$.  Then 
    lemma~\checkedref{Lem4Optimality1}
    applies to both $G$ and $G-3$, giving $v(G)=2$ and $v(G-3)=1$.
    The first of these is our claimed value for $v(G)$, and the second
    shows $v(G;3)=2$.  Since this equals $v(G)$, opening a 
    $3$-chain is optimal.
    Henceforth we 
    suppose $c(G)<-2$.
    
    Suppose $G$ is a counterexample with fewest possible components;
    we will derive a contradiction.
    The main step is to show $v(G;3)\in\{1,2\}$.
    Suppose first that $\theta=1$.  Lemma~\checkedref{Lem4Optimality2} 
    describes $v(G-3)$.  Because
    $c(G-3)\cong\size(G-3)\not\cong0$ mod~$4$, it says
    $v(G-3)=1$, $2$ or~$3$.  Then $v(G;3)=2$, $1$ or~$2$ respectively.
    Now suppose $\theta=2$ and $\size(G)\cong2$ mod~$4$.  Then 
    lemma~\checkedref{Lem4Optimality3}
    gives $v(G-3)\in\{1,3\}$, hence $v(G;3)=2$.  Finally, suppose
    that either $\theta=2$ and $\size(G)\not\cong2$ mod~$4$, or 
    that $\theta>2$.  Then the minimality of $G$ shows that the current
    lemma describes $v(G-3)$.  In particular, $v(G-3)\in\{1,2\}$, so
    $v(G;3)\in\{2,1\}$ also.
    This finishes the proof of $v(G;3)\in\{1,2\}$.

    Since $G$ is a counterexample, it has a component $C$ with
    $v(G;C)<v(G;3)$.  This is obviously impossible if $v(G;3)=1$, so
    we must have $v(G;3)=2$ and $v(G;C)=0$.  The latter forces 
    $C$ to be a $4$-loop with $v(G-4_\ell)=4$.  
    If $c(G-4_\ell)<-1$ then induction would give
the contradiction $v(G-4_\ell)\in\{1,2\}$.
Therefore 
$c(G-4_\ell)\geq-1$.  On the other hand,
$c(G-4_\ell)=c+4<2$, so we must have $c(G-4_\ell)=-1,0$ or~$1$.
There must be no $4$-loops in
    $G-4_\ell$, or else 
    lemma~\checkedref{Lem4Optimality1}
    would give the contradiction 
    $v(G-4_\ell)=|c(G-4_\ell)|=(\hbox{$0$ or $1$})\neq4$.
    But then lemma~\checkedref{Lem3Optimality1} 
    gives the contradiction $v(G-4_\ell)\in\{1,2,3\}$.
\end{proof}

\begin{proof}[Proof of theorem~\checkedref{ThmOpener}]
    This amounts to combining the results above.
    First suppose $c\geq2$.  
If 
    $G=3+{ }$(one or more loops) then
    theorem~\checkedref{ThmCAtLeast2}
    shows that opening the shortest loop is optimal.  
    Otherwise, the same theorem shows that  
    the standard move (opening a $3$-chain)
    is optimal.  
    
    Henceforth we may assume $c<2$.
    If $G$ has no $3$-chains or $4$-loops then
    we must prove that the standard move is optimal.
    Lemma~\checkedref{Lem6Loptimality} shows that $G$ has a $6$-loop,
    and that opening it (which is the standard move)
    is optimal.  
    If $G$ has  a $3$-chain but no $4$-loops
    then we must prove that the standard move (opening a $3$-chain)
    is optimal.  This is part of lemma~\checkedref{Lem3Optimality1}.
    If $G=4_\ell+3+3+3$ then again we must prove that the standard
    move (opening a $3$-chain) is optimal.  This is part of
    lemma~\checkedref{Lem3sAnd4Ls}.  

    We have reduced to the case that $c<2$ and that 
    $G\neq4_\ell+3+3+3$ has
    a $4$-loop.
    If $c\in\{0,\pm1\}$ then  
    lemma~\checkedref{Lem4Optimality1}  shows
    that opening a $4$-loop is optimal.

    Finally, suppose $c\leq-2$.  If $G$ has 
    no $3$-chains then we must prove that the standard move
    (opening a $4$-loop) is optimal.  This is part of 
    lemma~\checkedref{Lem4Optimality2}.  The only remaining cases
    are $G=3+4_\ell+H$.  If $H$ has no $3$-chains and $4|\size(H)$
    then we must show that opening a $4$-loop is optimal, which is
    part of lemma~\checkedref{Lem4Optimality3}.  On the other hand,
    if $H$ has a $3$-chain or $4\nmid\size(H)$ then we must show that
    the standard move (opening the $3$-chain) is optimal.  This
    is part of lemma~\checkedref{Lem3Optimality2}.
\end{proof}

\section{Game values and controller strategy}
\label{SecValues}

\noindent
The controller's strategy in theorem~\ref{ThmController} depends on
being able to recognize when a position~$G$ has value${}>2$.
We will give two ways to compute 
the value.  First we give an explicit value in terms of 
$c(G)$, $f(G)$ and the overall size and shape of $G$.  
This is complicated.  Then we give our
preferred method, whose essential case
computes $v(G)$ by starting with $c(G_0)$ for a certain
smaller position~$G_0$, and then applying a simple process.  
The following result immediately implies theorem~\ref{ThmVbiggerThan2}.

\begin{theorem}[Values---explicit]
    \label{ThmValuesExplicit}
    Suppose $G$ is a nonempty Dots~\& Boxes position consisting of loops
    and long chains.  Then its value $v=v(G)$ is given by the first of the
    following cases that applies:
    \begin{enumerate}
        \item
            \label{ItemValExplicitc}
            If $c\geq2$ then $v=c$.
        \item
            \label{ItemValExplicit0}
            If $c=0$ and $G=4_\ell+(\hbox{anything except $3+3$})$,
            then $v=0$.
        \item
            \label{ItemValExplicitComplicated}
            If $\theta=0$, 
            or if \,$\theta=1$ and
            $\size(G)\cong3$ mod~$4$, then
            \begin{enumerate}
                \item
                    if $c+4f\geq2$ then $v=0,1,2,3,4$ according to 
                    whether
                    $c\cong 0,\pm1,\pm2,\pm3,4$ mod~$8$.
                \item
                    if $c+4f<2$ then
                    \begin{enumerate}
                        \item
                            if $G$ is odd then $v=1$ resp.\ $3$ if 
                            $f$ is odd resp.\ even.
                        \item
                            if $\size(G)\cong2$ mod~$4$ then $v=2$.
                        \item
                            if $\size(G)\cong0$ mod~$4$ then
                            $v=0$ resp.\ $4$ if $f$ is odd resp.\ even.
                    \end{enumerate}
            \end{enumerate}
        \item
            \label{ItemValExplicitGeneric}
            In all other cases, $v=1,2$ as $G$ is odd,even.
    \end{enumerate}
\end{theorem}

\begin{proof}
    \eqref{ItemValExplicitc}
    If $c\geq2$ then $v=c$ by theorem~\ref{ThmCAtLeast2}.  So
    suppose $c<2$.

    \eqref{ItemValExplicit0}
    If $c=0$ and $f>0$ and $G\neq3+3+4_\ell$ then lemma~\ref{Lem4Optimality1}
    gives $v=0$, as we are asserting.

    \eqref{ItemValExplicitComplicated}
    Now suppose $\theta=0$, 
    or that $\theta=1$ and $\size(G)\cong3$ mod~$4$.
    If $c+4f\geq2$ then the theorem's value for $v$ is justified
    by lemma \checkedref{Lem4Optimality2} (if $\theta=0$)
    or~\checkedref{Lem4Optimality3} (if $\theta=1$).
    The $c+4f<2$ case is also justified by these lemmas.

    \eqref{ItemValExplicitGeneric}
    In all  remaining cases we have $\theta\geq2$, or 
    $\theta=1$ and $\size(G)\not\cong3$ mod~$4$.  We must prove 
    $v(G)\in\{1,2\}$.  
    If $G$ has no $4$-loop then we refer to lemma~\checkedref{Lem3Optimality1}.
    So we suppose $f>0$.  If $c=\pm1$ then $v=1$ by 
    lemma~\checkedref{Lem4Optimality1}.  If $c=0$ then we must
    have $G=4_\ell+3+3$, or else case~\eqref{ItemValExplicit0} would
    have applied.  This special case has value~$2$
    by lemma~\checkedref{Lem3sAnd4Ls}.  Finally, if $c<-1$ then
    we refer to lemma~\checkedref{Lem3Optimality2}.
\end{proof}

There is a way to replace the complicated case \eqref{ItemValExplicitComplicated} by
a simpler iterative procedure.  In practice it is easier and
less error-prone than working through the tree of subcases.
It is quicker to use than explain, so the reader might want
to look ahead at example~\ref{EgProcedural}.
We will define a union $G_0$ of components of~$G$.
The strategy for computing $v(G)$ is to show $v(G_0)=c(G_0)$ and
that $v(G)$
is got from $v(G_0)$ by a simple process.  

Suppose  $\theta\leq1$.  
We define the \emph{core}
$G_0$ of $G$
as the union of the following components of~$G$.
Its loops are all the loops of length${}\geq8$.
Its chains are either all the chains of length${}\geq4$
(if there are any) or the $3$-chain 
(if $G$ has one and no other chains).
The core has the useful property that if $H$ is the union of
$G_0$ and possibly some more components of $G$, then $H_0=G_0$.
Write $\theta'\in\{0,1\}$ for the number of $3$-chains in $G-G_0$.

We define maps $\Z\to\Z$:
\begin{equation*}
    \Sigma(x){}=|x-4|+2
    \qquad\quad
    \Theta(x){}=x-1
    \qquad\quad
    \Phi(x){}=|x-4|
\end{equation*}
Only their values on $\Z_{\geq0}$ will be important.  
We think  of $\Sigma$ as decreasing any $x\geq2$ in steps of size~$2$ 
until $x$ enters
$[2,4]$, where $\Sigma$ acts as reflection across~$3$.  
And $\Phi$ decreases $x$ in
steps of size~$4$
until $x$ enters $[0,4]$, where $\Phi$ acts as reflection across~$2$.
These
operations are from \cite{BC}, but the generality in which 
we use them is new.

\begin{theorem}[Values---procedural]
    \label{ThmValuesProcedural}
    Suppose $G$ is a nonempty Dots \& Boxes position consisting
    of loops and long chains.  Then
    \begin{enumerate}
        \item
            \label{ItemValProcc}
            if $c\geq2$ then $v=c$;
        \item
            \label{ItemValProc0}
            if $c=0$, and $G$ has a $4$-loop, and $G\neq4_\ell+3+3$, then $v=0$;
        \item
            \label{ItemValProcOperators}
            if $\theta=0$, or if $\theta=1$ and $\size(G)\cong3$ $\mathrm{mod}\ 4$, then
            $$v=\Phi^f\Theta^{\theta'}\Sigma^s(c(G_0))$$
    \end{enumerate}
    In all other cases $v=1$ or~$2$ according to the parity of $G$.
\end{theorem}

\begin{example}
    \label{EgProcedural}
    Say $G=(8_\ell^2+18)+6_\ell^9+3+4_\ell^{101}$, where the parentheses
    indicate~$G_0$.  Because $\theta=1$ we check that $\size(G)\cong3$
    mod~$4$, so the theorem applies.
    We begin with $v(G_0)=c(G_0)=18$.  Adjoining the first seven  
    $6$-loops decreases this by~$2$ each time, leaving $4$.  Adjoining
    the remaining two
    $6$-loops bounces this to $2$ and back to~$4$ (reflecting across~$3$).
    Adjoining the $3$-chain reduces this to~$3$.  Adjoining 
    the $4$-loops bounces this between $3$ and $1$ an odd number of
    times, leaving $v(G)=1$.
\end{example}

\begin{proof} 
    Unlike theorem~\ref{ThmValuesExplicit}, no case  \eqref{ItemValProcc}--\eqref{ItemValProcOperators} 
    has priority over any other.  
    When two apply then we are asserting that both are correct.
    Cases \eqref{ItemValProcc} and \eqref{ItemValProc0} are justified by theorem~\ref{ThmCAtLeast2}
    and lemma~\ref{Lem4Optimality1} respectively.  The ``in all other cases'' argument
    is the same as for theorem~\ref{ThmValuesExplicit}.  

    For  case \eqref{ItemValProcOperators}
    we 
    first prove $v(G_0)=c(G_0)$.  If $G_0=\emptyset$ 
    then both sides 
    are~$0$, so suppose $G_0\neq\emptyset$.
    Every component of $G_0$ is a loop of length${}\geq8$ or a chain
    of length${}\geq4$, or a $3$-chain.  There is at most
    one $3$-chain.  Every term in the formula \checkedeqref{EqCV} for
    $c(G_0)$ is nonegative, except that one term can be~$-1$.  So
    $c(G_0)\geq\tb(G_0)-1\geq3$, which implies $v(G_0)=c(G_0)$.

    Next we observe that 
    the initial moves in
    an optimal strategy for the opener is to open
    all the $4$-loops, then the $3$-chain (if $G-G_0$ has one),
    then the $6$-loops.  This follows from theorem~\checkedref{ThmOpener}.
    Therefore $v(G)$ can be got by working backwards from 
    $v(G_0)=c(G_0)$.
    We defined $\Sigma$ so that $v(H)=\Sigma(v(H-6_\ell))$
    for any position $H$ in which opening a $6$-loop is optimal.
    It follows that if $G-G_0$ has no $3$-chain or $4$-loops, 
    then $v(G)=\Sigma^s(v(G_0))$.  If $G-G_0$ has a $3$-chain
    but no $4$-loops, then instead we use 
    $v(G)=1+|v(G-3)-2|=\Theta(v(G-3))$.  The second inequality
    here comes from the fact that $G_0\sset G-3$ is nonempty 
    (else the
    $3$-chain would lie in $G_0$) without $3$-chains or $4$-loops,
    hence has value${}\geq2$ by lemma~\ref{Lem6Loptimality}.
    We have proven $v=\Theta^{\theta'}\Sigma^s(c_0(G))$.
    For general $G$ the same argument shows
    \begin{align*}
        v(G)&{}=\Phi^f\bigl(v(G)-(\hbox{all $4$-loops})\bigr)
    =\Phi^f\Theta^{\theta'}\Sigma^s(v(G_0)).\qedhere
    \end{align*}
\end{proof}

\section{Consequences for midgame play}
\label{SecMidgame}

\noindent
This section 
discusses how to use our results in actual play.   After 
reaching a position consisting of loops and long chains,
one can play optimally by following the strategies in the introduction.
But before play gets that far, there are opportunities
to influence the shape of the endgame.  

We write from the perspective of a player (``you'') 
who expects to be the opener once the game settles down to a union
of loops and long chains.  This means that you expect to lose the
Nimstring game.  
Assuming you are right in this expectation, your
only chance of victory is to gain a large enough 
advantage in captured boxes before the endgame.
For a line of play under consideration, we 
write $G$ for the resulting endgame and  $A$ for your advantage
in boxes at the time play reaches that position.  
Players prefer  odd${}\times{}$odd boards because ties are impossible,
so we restrict to this case.  As a consequence, $A$ and  $\size(G)$ 
have different parities.  
We also assume that you will create 
enough $3$-chains and $4$- and $6$-loops  so that  $c(G)<2$
and hence $v(G)\in\{0,\dots,4\}$.  

A simple rule of thumb is: if $A\leq1$ then you will lose, while 
if $A\geq2$ then you will win.  This is not really true, but it
is true ``generically'' in the sense that  theorem~\ref{ThmValuesExplicit}
shows that almost all positions
with $c(G)<2$ have value $1$ or~$2$.
If $A=1$ then the only way you 
can win is to arrange for $v(G)=0$.  
If $A=2$ resp.~$3$ then the only way you
can lose is for $v(G)=3$ resp.~$4$.  
According to theorem~\ref{ThmValuesExplicit}, 
each of these imposes very strong constraints on~$G$.
We now consider in more detail
the possibilities for pairs $(A,G)$.  Your goal, for given $A$, is to
steer the game toward an endgame $G$ satisfying $A-v(G)>0$.
We refer to theorem~\ref{ThmValuesExplicit} throughout the analysis.
Any $(A,G)$ with $A\leq0$ is lost, 
while
any $(A,G)$ with $A\geq4$ is a win.

\smallskip
First consider $(A,G)$ with  $A=1$.  
You have probably lost, but if you can arrange $v(G)=0$
then you will win.
Your opponent will try to block the creation of a $4$-loop,
and if she succeeds then you will lose ($f=0$ implies $v\neq0$, hence
$v\geq2$).
So we
assume that you can arrange for a $4$-loop to be present.
You will win if you can also arrange for $c=0$,
with the single exception $G=4_\ell+3+3$.

Our advice past this point is less likely to be useful because it
gets into the complicated case \eqref{ItemValExplicitComplicated} of
theorem~\ref{ThmValuesExplicit}.  
If you cannot arrange for $c=0$, then you must arrange for $\theta=0$
and $\size(G)\cong0$ mod~$4$, so
we assume you can achieve this.  (By $c\neq0$, $\theta>0$ 
implies  $v>0$.  And if $\theta=0$ then $v=0$ forces $\size(G)\cong0$
mod~$4$.)
You will win if you can also arrange for the number $s$ of $6$-loops
to be small and
$c\cong0$ mod~$8$, or if you can arrange for $s$ to be large and $f$ to
be even.  Here and below we  write ``$s$ small'' and ``$s$ large'' as
stand-ins for the more precise statements $c+4f\geq2$ and $c+4f<2$
respectively.  (Because $\theta=0$, making $c+4f$ small requires
$6$-loops.)  Actual play is unlikely to lead to many $6$-loops,
so one could ignore the ``$s$ large'' case without much loss.

\smallskip
Now consider $(A,G)$ with $A=2$.  You will probably win,
but if your opponent can arrange for $v(G)=3$ then you will lose.
If you can create
two $3$-chains then you will win.  Even a single $3$-chain will
do if $\size(G)\not\cong3$ mod~$4$.  If you cannot acheive this, then
you must arrange for both of the second pair of alternatives 
of theorem~\checkedref{ThmVbiggerThan2} to 
fail.  That is, you must arrange for $s$ to be small and $c(G)\cong\pm1$
mod~$8$, or for $s$ to be large and $f$ to be odd.

\smallskip
Finally consider $(A,G)$ with $A=3$.  This is similar to
the $A=2$ case but easier since you will win unless $v(G)=4$.
If you arrange for $G$ to have a
$3$-chain, or $\size(G)\cong2$ mod~$4$, 
then you will win.  (This uses
the evenness of $G$.)  Suppose you cannot acheive either of
these, so $\theta=0$ and $\size(G)\cong0$ mod~$4$.
Then you must aim for $s$ to be small and $c(G)\cong0$ mod~$8$,
or for $s$ to be large and $f$ to be odd.

\end{document}